\documentclass[10pt]{amsart}

\usepackage{amssymb,latexsym,amsmath,epsfig,amsthm} 
\usepackage[utf8]{inputenc} 
\usepackage{amsbsy}
\usepackage{amscd}
\usepackage{verbatim}
\usepackage[english]{babel}
\usepackage{dsfont}
\usepackage{anysize}
\usepackage{hyperref}
\usepackage{leftidx}
\usepackage{footnote}
\usepackage{tablefootnote}

\newcommand{\C}{\mathds{C}}

\newcommand{\q}{\quad}

\newtheorem{thm}{Theorem}[section]
\newtheorem{lem}[thm]{Lemma}

\newtheorem{prop}[thm]{Proposition}
\newtheorem{kor}[thm]{Corollary}

\theoremstyle{definition}
\newtheorem*{defi}{Definition}

\theoremstyle{remark}
\newtheorem*{rema}{Remark}

\title{New estimates for some integrals of functions defined over primes}

\author{Christian Axler}
\email{christian.axler@hhu.de}

\date{\today}

\begin{document}

\begin{abstract}
In this paper we give new estimates for integrals involving some arithmetic functions defined over prime numbers. The main focus here is on the prime counting function $\pi(x)$ and the Chebyshev $\vartheta$-function. Some of these estimates depend on the correctness of the Riemann hypothesis on the nontrivial zeros of the Riemann zeta function $\zeta(s)$. 
\end{abstract}

\maketitle

\section{Introduction}

The Riemann zeta function is for all complex numbers $s$ with $\text{Re}(s) > 1$ defined as
\begin{displaymath}
\zeta(s) = \sum_{n=1}^s \frac{1}{n^s}.
\end{displaymath}
It is well known that the Riemann zeta function is a meromorphic function on the whole complex plane, which is holomorphic everywhere except for a simple pole at $s = 1$ with residue 1. Euler discovered a remarkable connection between the Riemann zeta function and the prime numbers by showing that
\begin{displaymath}
\zeta(s) = \prod_{p} \frac{1}{1-p^{-s}} \q\q (\text{Re}(s) > 1),
\end{displaymath}
where $p$ runs over all primes. Therefore, the Riemann zeta function plays a special role in analytic number theory. The Riemann zeta function satisfies the functional equation
\begin{displaymath}
\zeta(s)= 2^s \pi^{s-1} \sin \left({\frac{\pi s}{2}} \right) \Gamma(1-s)\zeta (1-s),
\end{displaymath}
where $\Gamma(s)$ is the gamma function. This is an equality of meromorphic functions valid on the whole complex plane. Due to the zeros of the sine function, the functional equation implies that $\zeta(s)$ has outside the set $\{ s \in \C \mid 0 \leq \text{Re}(s) \leq 1 \}$ a simple zero at each even negative integer $s = -2n$, known as the \textit{trivial zeros} of the Riemann zeta function. The \textit{nontrivial zeros}, i.e. the zeros in the set $\{ s \in \C \mid 0 \leq \text{Re}(s) \leq 1 \}$, have attracted far more attention because not only is their distribution far less well known, but their study also yields important results concerning primes and related objects in number theory.
The \textit{Riemann hypothesis} asserts that the real part of every nontrivial zero of the Riemann zeta function is $1/2$. To this day, the Riemann hypothesis is considered one of the greatest unsolved problems in mathematics. In this paper we will derive effective estimates for some integrals of functions defined over primes. Some of these estimates are based on the assumption that the Riemann hypothesis is true. First, we consider the prime counting function $\pi(x)$ which gives the number of primes not exceeding $x$. Hadamard \cite{hadamard1896} and de la Vall\'{e}e-Poussin \cite{vallee1896} independently proved a result concerning the asymptotic behavior for $\pi(x)$, namely
\begin{equation}
\pi(x) \sim \text{li}(x) \q\q (x \to \infty), \tag{1.1} \label{1.1}
\end{equation}
which is known as the \textit{Prime Number Theorem}. Here, the \textit{integral logarithm} $\text{li}(x)$ is defined by
\begin{displaymath}
\text{li}(x) = \int_0^x \frac{\text{d}t}{\log t} = \lim_{\varepsilon \to 0+} \left \{ \int_{0}^{1-\varepsilon}{\frac{\text{d}t}{\log t}} + 
\int_{1+\varepsilon}^{x}{\frac{\text{d}t}{\log t}} \right \}.
\end{displaymath}
In a later paper \cite{vallee1899}, where the existence of a zero-free region for the Riemann zeta function to the left of the line $\text{Re}(s) = 1$ was proved, de la Vall\'{e}e-Poussin also estimated the error term in the Prime Number Theorem by showing that
\begin{equation}
\pi(x) = \text{li}(x) + O(x e^{-c_0\sqrt{\log x}}) \q\q (x \to \infty), \tag{1.2} \label{1.2}
\end{equation} 
where $c_0$ is a positive absolute constant. In 1793, Gauss computed that
\begin{equation}
\pi(x) \leq \text{li}(x) \tag{1.3} \label{1.3}
\end{equation}
holds for every $x$ with $2 \leq x \leq 3,000,000$ and conjectured that the inequality \eqref{1.3} holds for every $x \geq 2$. This conjecture was disproven by Littlewood \cite{littlewood} by showing that the function $\pi(x) - \text{li}(x)$ changes the sign infinitely many times as $x$ increases to infinity. More precisely, he proved that
\begin{displaymath}
\pi(x) - \text{li}(x) = \Omega_{\pm} \left( \frac{\sqrt{x} \log \log \log x}{\log x} \right). \tag{1.4} \label{1.4}
\end{displaymath}
Unfortunetely, Littlewood’s proof is nonconstructive and there is still no example of $x$ such that $\pi(x) > \text{li}(x)$. Under the assumption that the Riemann hypothesis on the nontrivial zeros of $\zeta(s)$ is true, several authors (see, for instance, Ingham \cite{ingham}, Prachar \cite{prachar}, and Kaczorowski \cite{kaczorowski}) asserted that
\begin{equation}
\int_2^x (\pi(t) - \text{li}(t)) \, \text{d}t < 0 \tag{1.5} \label{1.5}
\end{equation}
for all sufficiently large values of $x$. Stechkin and Popov \cite[Corollary 10]{stechkin} found that under the assumption that the Riemann hypothesis is true, one has
\begin{equation}
- \frac{0.714x^{3/2}}{\log x} < \int_2^x (\pi(t) - \text{li}(t)) \, \text{d}t < - \frac{0.62x^{3/2}}{\log x} \tag{1.6} \label{1.6}
\end{equation}
for all sufficiently large values of $x$. Pintz \cite{pintz} stated without proof that the inequality \eqref{1.5} is even a sufficient condition for the truth of the Riemann hypothesis. Recently, Johnston \cite[Theorem 1.1]{johnston2} was able to show the following Proposition.

\begin{prop}[Johnston] \label{prop101}
The Riemann hypothesis is true if and only if
\begin{equation}
\int_2^x (\pi(t) - \emph{li}(t)) \, \emph{d}t < 0 \q\q (x > 2). \tag{1.7} \label{1.7}
\end{equation}
\end{prop}
 
In this paper, we first give the following result. Here, let $\psi(x) = \sum_{p^m \leq x} \log p$ denote Chebyshev's $\psi$-function and for $y \geq 2$ let
\begin{displaymath}
d_0(y) = \sum_{m=1}^{\lfloor \frac{\log y}{\log 2} \rfloor} \frac{\pi(y^{1/m})}{m} - \text{li}(y) - \frac{\psi(y)-y}{\log y} - \frac{1}{y \log^2 y} \int_{2}^y (\psi(t)-t) \, \text{d}t.
\end{displaymath}

\begin{thm} \label{thm102}
Let $c_0 = \log 2\pi$, $c_1 = 2 + (\zeta'/\zeta)(-1)$, $s=17^2$, and $\omega = \sum_{\rho} |\rho(\rho+1))|^{-1}$ where $\rho$ runs over all all nontrivial zeros of the Riemann zeta function. Further, let $\delta$ be a real positive number with $\delta \geq \omega$ and let
\begin{align*}
R_1(y) & = \int_{2}^y (\psi(t)-t) \, \emph{d}t, \\
f_1(x, \delta) & = \left( \frac{3\delta}{4}  - 1 \right) \emph{li}(x^{3/2}) - \frac{3\delta}{4} \, x \emph{li}(\sqrt{x}), \\
g(x) & = - 1.2762 \left( 3 \, \emph{li}(x^{3/2}) - \frac{2x^{3/2}}{\log x} + 4 \, \emph{li}(x^{4/3}) - \frac{3x^{4/3}}{\log x} \right) - \emph{li}(x^{4/3}) - \frac{5.1048x^{5/4}}{5\log 2}, \\ 
J_0(\delta) & = d_0(2) + \delta \left( \frac{3}{4} \, \emph{li}(\sqrt{2}) - \frac{\sqrt{2}(3\log 2+2)}{2 \log^2 2} \right) - c_0 \left( \frac{1}{\log 2} + \frac{1}{\log^2 2} \right), \\
J_1(\delta) & = \emph{li}(2^{3/2}) + g(2) - 2 j_0(2, \delta) - \frac{R_1(2)}{\log 2} - \frac{2c_0 + 2^{3/2}\delta}{\log 2} + \frac{3\delta}{4} (2 \emph{li}(\sqrt{2}) - \emph{li}(2^{3/2})), \\
f_2(x, \delta) & = \frac{3\delta}{4} \, x \emph{li}(\sqrt{x}) - \left( 1 + \frac{3\delta}{4} \right) \emph{li}(x^{3/2})), \\
J_2(\delta) & = d_0(s) - \delta \left( \frac{3}{4} \, \emph{li}(\sqrt{s}) - \frac{\sqrt{s}(3\log s+2)}{2 \log^2 s} \right) - c_0 \left( \frac{1}{\log s} + \frac{1}{\log^2 s} \right) + \frac{c_1}{s\log^2s}, \\
J_3(\delta) & = \emph{li}(s^{3/2}) + \int_2^s (\pi(t) - \emph{li}(t)) \, \emph{d}t -J_2(s, \delta) s - \frac{3\delta}{4} (s \emph{li}(\sqrt{s}) - \emph{li}(s^{3/2})) - \frac{R_1(s)}{\log s} + \frac{c_1 - c_0s + \delta s^{3/2}}{\log s}.
\end{align*}
Then, under the assumption that the Riemann hypothesis is true, we have
\begin{displaymath}
J_1(\delta) + J_0(\delta)x + f_1(x, \delta) + g(x) < \int_2^x (\pi(t) - \emph{li}(t)) \, \emph{d}t < J_3(\delta) + J_2(\delta) x + f_2(x, \delta),
\end{displaymath}
where the left-hand side inequality holds for every $x \geq 2$ and the right-hand side inequality holds for every $x \geq s$. 
\end{thm}

Since neither a closed formula nor the exact value of $\omega = \sum_{\rho} |\rho(\rho+1))|^{-1}$ is known, we can combine a well known upper bound for $\omega$ (see \eqref{2.12}) with the inequalities given in Theorem \ref{thm102} to establish the following effective estimates for the integral in \eqref{1.7} which on the one hand give an explicit version of \eqref{1.6} and on the other hand provide a more precise suffient and necessary criterion for the Riemann hypothesis compared to Proposition \ref{prop101}.

\begin{kor} \label{kor103}
Let $\lambda_0 = 2 + \gamma - \log 4\pi = 0.0461914\ldots$. The Riemann hypothesis is true if and only if
\begin{displaymath}
- \frac{x^{3/2}}{\log x} < \int_2^x (\pi(t) - \emph{li}(t)) \, \emph{d}t < \left( \lambda_0 - \frac{2}{3} \right) \times \frac{x^{3/2}}{\log x},
\end{displaymath}
where the left-hand side inequality holds for every $x \geq 2,258,093,575$ and the right-hand side inequality holds for every $x \geq 139,537,375$. In particular, the right-hand side inequality in \eqref{1.6} holds for every $x \geq 4.82 \times 10^{863}$.
\end{kor}

\begin{rema}
Note that the positive integer $N_0 = 2,258,093,575$ might not be the smallest positive integer so that the left-hand side inequality in Corollary \ref{kor103} holds for every $x \geq N_0$. The same statement holds also for $N_1 = 139,537,375$ and $N_2 = 4.82 \times 10^{863}$.
\end{rema}

Although Littlewood \cite{littlewood} could exhibit that the function $\pi(x) - \text{li}(x)$ changes its sign infinitely often as $x$ increases to infinity, it can be shown that certain mean values of $\pi(x) - \text{li}(x)$ are negative even without the assumption that the Riemann hypotesis is true. Pintz \cite{pintz} utilized standard complex analysis methods to prove that
\begin{displaymath}
\int_1^{\infty} (\pi(t) - \text{li}(t)) \exp \left( - \frac{\log^2t}{y} \right) \, \text{d}t < 0
\end{displaymath}
for all sufficiently large values of $y$. Using an upper bound for the sum of the reciprocals of all prime numbers $p$ not exceeding $x$,
Johnston \cite[p.\:7]{johnston2} was recently able to show that
\begin{equation}
\int_2^x \frac{\pi(t) - \text{li}(t)}{t^2} \, \text{d}t < \frac{2.5}{\log^2 x} - 0.62 \tag{1.8} \label{1.8}
\end{equation}
for every $x > 200$. It follows in particular that the integral in \eqref{1.8} is negative for every $x > 2$. In the next theorem, we also utilize effective estimates for the sum of the reciprocals of all prime numbers $p$ not exceeding $x$ to improve the inequality \eqref{1.8} on the one hand and to determine a lower bound for the integral in \eqref{1.8} on the other hand.

\begin{thm} \label{thm104}
Let $B$ be \textit{the Mertens' constant}; i.e.
\begin{equation}
B = \gamma + \sum_{p} \left( \log \left( 1 - \frac{1}{p} \right) + \frac{1}{p} \right) = 0.26149 \ldots , \tag{1.9} \label{1.9}
\end{equation}
where $\gamma = 0.577215\ldots$ denotes the Euler-Mascheroni constant, and let $C$ be defined by
\begin{equation}
C = B - \frac{\emph{li}(e)}{e} - \int_2^e \frac{\emph{li}(t)}{t^2} \, \emph{d}t = -0.62759\ldots. \tag{1.10} \label{1.10}
\end{equation}
Then, for every $x \geq 1,757,126,630,797$, we have
\begin{displaymath}
C - \frac{0.0100757}{\log^3 x} < \int_2^x \frac{\pi(t) - \emph{li}(t)}{t^2} \, \emph{d}t < C + \frac{0.0101517}{\log^3 x}.
\end{displaymath}
In particular, one has
\begin{displaymath}
\int_2^{\infty} \frac{\pi(t) - \emph{li}(t)}{t^2} \, \emph{d}t = C.
\end{displaymath}
\end{thm}

After Johnston \cite[Theorem 1.3]{johnston2} has shown that
\begin{equation}
\int_2^x \frac{(\pi(t) - \text{li}(t))}{t^2} \, \text{d}t < 0 \tag{1.11} \label{1.11}
\end{equation}
for every $x > 2$, it is natural to ask whether the same statement is also true for the integral
\begin{equation}
\int_2^x \frac{(\pi(t) - \text{li}(t))}{t^c} \, \text{d}t, \tag{1.12} \label{1.12}
\end{equation}
where $c$ is a real number with $c < 2$. Here, Johnston \cite[Theorem 1.4]{johnston2} showed that one cannot do much better than \eqref{1.11} without further knowledge of the location of the nontrivial zeros of the Riemann zeta function. More precisely, he set $\omega = \sup \{ \text{Re}(s) \mid \zeta(s) = 0 \}$ and showed under the assumption that $1/2 < \omega \leq 1$ and $c < 1 + \omega$ that there exist arbitrarily large values of $x$ such that the integral \eqref{1.12} is positive. In the same paper, Johnston \cite[p.\:3]{johnston2} asked if it was possible to use a slightly asymptotically larger weight than $f(t) = 1/t^2$ in \eqref{1.11} and proposed $f(t) = \log(t)/t^2$. In the following theorem, we give effective estimates for the integral
\begin{equation}
\int_2^x \frac{(\pi(t) - \text{li}(t)) \log t}{t^2} \, \text{d}t \tag{1.13} \label{1.13}
\end{equation}
and furthermore determine its limit value for $x \to \infty$.

\begin{thm} \label{thm105}
Let $B$ be defined as in \eqref{1.9} and let
\begin{equation}
E = - \gamma - \sum_{p} \frac{\log p}{p(p-1)} = -1.3325 \ldots. \tag{1.14} \label{1.14}
\end{equation}
Further, we set
\begin{displaymath}
K = B + E + \log 2 + \log \log 2 - \frac{(1+\log 2)\emph{li}(2)}{2} = -1.62925885667\ldots.
\end{displaymath}
Then for every $x \geq 1,757,126,630,797$, one has
\begin{displaymath}
K - \frac{0.014262}{\log^2 x} < \int_2^x \frac{(\pi(t) - \emph{li}(t)) \log t}{t^2} \, \emph{d}t < K + \frac{0.014352}{\log^2 x}.
\end{displaymath}
In particular, we have
\begin{displaymath}
\int_2^{\infty} \frac{(\pi(t) - \emph{li}(t)) \log t}{t^2} \, \emph{d}t = K.
\end{displaymath}
\end{thm}

We get the following corollary which states that the integral \eqref{1.13} is indeed negative for every $x \geq 2$.

\begin{kor} \label{kor106}
For every $x > 2$, we have
\begin{displaymath}
\int_2^{\infty} \frac{(\pi(t) - \emph{li}(t)) \log t}{t^2} \, \emph{d}t < 0.
\end{displaymath}
\end{kor}

Next, we consider Chebyshev's $\vartheta$-function
\begin{displaymath}
\vartheta(x) = \sum_{p \leq x} \log p,
\end{displaymath}
where $p$ runs over all primes not exceeding $x$. Chebyshev's $\vartheta$-function and the prime counting function $\pi(x)$ are connected by the identities
\begin{align}
\pi(x) &= \frac{\vartheta(x)}{\log x} + \int_{2}^{x}{\frac{\vartheta(t)}{t \log^{2} t}\ \text{d}t}, \tag{1.15} \label{1.15} \\
\vartheta(x) & = \pi(x) \log x - \int_{2}^{x}{\frac{\pi(t)}{t}\ \text{d}t}, \tag{1.16} \label{1.16}
\end{align}
which holds for every $x \geq 2$ (see \cite[Theorem 4.3]{ap}). If we combine \eqref{1.2} and \eqref{1.16}, we see that
\begin{displaymath}
\vartheta(x) = x + O(x e^{-c_1\sqrt{\log x}}) \q\q (x \to \infty),
\end{displaymath} 
where $c_1$ is a positive absolute constant. Similar to \eqref{1.4}, Littlewood \cite{littlewood} showed that
\begin{displaymath}
\vartheta(x) - x = \Omega_{\pm}( \sqrt{x} \log \log \log x)
\end{displaymath}
which implies that the function $\vartheta(x) - x$ changes the sign infinitely many times as $x$ increases to infinity. Nonetheless, it can be shown that certain mean values of $\vartheta(x) - x$ are negative. Similar to Proposition \ref{prop101}, Johnston \cite[Theorem 1.2]{johnston2} showed that the negativity of the integral
\begin{equation}
\int_2^x (\vartheta(t) - t) \, \text{d}t \tag{1.17} \label{1.17}
\end{equation}
for all $x \geq 2$ is a necessary and sufficient condition for the truth of the Riemann hypothesis. In order to refine this equivalent criterion, we first give the following result.

\begin{thm} \label{thm107}
Let $\omega = \sum_{\rho} |\rho(\rho+1))|^{-1}$ where $\rho$ runs over all all nontrivial zeros of the Riemann zeta function. Then, under the assumption that the Riemann hypothesis is true, one has
\begin{displaymath}
-\left( \omega + \frac{2}{3} \right) \, x^{3/2} - x^{4/3} - x\log 2\pi < \int_2^x (\vartheta(t) - t) \, \emph{d}t < \left( \omega - \frac{2}{3} \right) \, x^{3/2} - x \log 2\pi,
\end{displaymath}
where the left-hand side inequality holds for every $x \geq 2$ and the right-hand side inequality holds for every $x \geq 121$.
\end{thm}

As in the proof of Corollary \ref{kor103}, we combine a well known upper bound for $\omega$ (see \eqref{2.12}) with the inequalities given in Theorem \ref{thm107} to establish the following more precise version of \eqref{1.17}.
 
\begin{kor} \label{kor108}
Let $\lambda_0 = 2 + \gamma - \log 4\pi = 0.0461914\ldots$. The Riemann hypothesis is true if and only if
\begin{displaymath}
- 0.713 \times x^{3/2} < \int_2^x (\vartheta(t) - t) \, \emph{d}t < \left( \lambda_0 - \frac{2}{3} \right) \, x^{3/2}, \tag{1.18} \label{1.18}
\end{displaymath}
where the left-hand side inequality holds for every $x \geq 1.224117 \times 10^{23}$ and the right-hand side inequality holds for every $x \geq 10$.
\end{kor}

\begin{rema}
Note that the positive integer $N_0 = 1.224117 \times 10^{23}$ might not be the smallest positive integer so that the left-hand side inequality given in \eqref{1.18} holds for every $x \geq N_0$.
\end{rema}

Without the assumption that te Riemann hypothesis is true, Johnston \cite[Equation (4.5)]{johnston2} found analogously to \eqref{1.8} that the inequality 
\begin{equation}
\int_2^x \frac{\vartheta(t) - t}{t^2} \, \text{d}t < \frac{1.5}{\log x} - 1.63 \tag{1.19} \label{1.19}
\end{equation}
holds for every $x > 200$. In particular, this inequality and a simple computation for smaller values of $x$ show that the integral in \eqref{1.19} is negative for every $x > 2$. Finally, we derive the following improved effective estimates for the integral in \eqref{1.19} and also determine its limit value for $x \to \infty$.

\begin{thm} \label{thm109}
Let $E$ defined as in \eqref{1.14} and let $D$ be defined by
\begin{equation}
D = \log 2 - 1 + E = -1.63943509\ldots. \tag{1.20} \label{1.20}
\end{equation}
Then, for every $x \geq 1,757,126,630,797$, one has
\begin{displaymath}
D - \frac{0.024334}{2\log^2 x} \left( 1 + \frac{4}{\log x} \right) < \int_2^x \frac{\vartheta(t) - t}{t^2} \, \emph{d}t < D + \frac{0.024334}{2\log^2 x} \left( 1 + \frac{4}{\log x} \right).
\end{displaymath}
In particular, we have
\begin{displaymath}
\int_2^{\infty} \frac{\vartheta(t) - t}{t^2} \, \emph{d}t = D. \tag{1.21} \label{1.21}
\end{displaymath}
\end{thm}

\section{Proof of Theorem \ref{thm102}}

Throughout this paper, let
\begin{displaymath}
R(t) = \psi(t) - t,
\end{displaymath}
where $\psi(x) = \sum_{p^m \leq x} \log p$ denotes Chebyshev's $\psi$-function, and
\begin{equation}
R_1(x) = \int_2^x R(t) \, \text{d}t. \tag{2.1} \label{2.1}
\end{equation}
In order to prove Theorem \eqref{thm102}, we first note the following lemma where we give a small refinement of \cite[Lemma 2.6]{johnston2}.

\begin{lem} \label{lem201}
Let $\omega = \sum_{\rho} |\rho(\rho+1)|^{-1}$, where $\rho$ runs over all all nontrivial zeros of the Riemann zeta function. Under the assumption that the Riemann hypothesis is true, we have
\begin{displaymath}
2 + \frac{\zeta'}{\zeta}(-1) - x \log 2\pi - \omega x^{3/2} - \frac{x}{2(x^2-1)} \leq R_1(x) < 2 + \frac{\zeta'}{\zeta}(-1) - x \log 2\pi + \omega x^{3/2}
\end{displaymath}
for every $x \geq 2$.
\end{lem}

\begin{proof}
Let
\begin{displaymath}
\delta(x) = \frac{x^2}{2} + \frac{\zeta'}{\zeta}(-1) - x \log 2\pi - \sum_{r=1}^{\infty} \frac{x^{1-2r}}{2r(2r-1)}.
\end{displaymath}
Since
\begin{displaymath}
0 \leq \sum_{r=1}^{\infty} \frac{x^{1-2r}}{2r(2r-1)} \leq \frac{x}{2(x^2-1)}
\end{displaymath}
for every $x>1$, we can see that the inequalities
\begin{equation}
\frac{x^2}{2} - x\log 2\pi + \frac{\zeta'}{\zeta}(-1) - \frac{x}{2(x^2-1)} \leq \delta(x) \leq \frac{x^2}{2} - x\log 2\pi + \frac{\zeta'}{\zeta}(-1) \tag{2.2} \label{2.2}
\end{equation}
hold for every $x \geq 2$. By Ingham \cite[Theorem 28]{ingham} (or Edwards \cite[p.\:74]{edwards}), one has
\begin{equation} 
\int_2^x \psi(t) \, \text{d}t - \delta(x) = -\sum_{\rho} \frac{x^{1 + \rho}}{\rho(\rho+1)} \tag{2.3} \label{2.3}
\end{equation}
for every $x \geq 2$, where $\rho$ runs over all all nontrivial zeros of the Riemann zeta function. Since we assumed that the Riemann hypothesis is true, we have $\text{Re}(\rho) = 1/2$. Hence, the equation \eqref{2.3} gives
\begin{equation}
\left| \int_2^x \psi(t) \, \text{d}t - \delta(x) \right| \leq \omega x^{3/2}. \tag{2.4} \label{2.4}
\end{equation}
Finally, it suffice to apply \eqref{2.2} to \eqref{2.4} and we arrive at the end of the proof.
\end{proof}

\begin{rema}
Since the function
\begin{displaymath}
\int_2^x (\psi(t)-t) \, \text{d}t
\end{displaymath}
changes its sign infinitely often as $x$ increases to infinity (see Johnston \cite[Lemma 2.8]{ingham}), a result similar to Proposition \ref{prop101} with $\psi(t) - t$ instead of $\pi(t) - \text{li}(t)$ does not hold.
\end{rema}

\begin{rema}
In Corollary \ref{kor603}, we find effective estimates for the integral
\begin{displaymath}
\int_2^{\infty} \frac{\psi(t) - t}{t^2} \, \text{d}t.
\end{displaymath}
\end{rema}

Next, we introduce the following auxiliary function.

\begin{defi}
For $t \geq 2$, let $M(t) = \lfloor \log(t)/\log(2) \rfloor$ and let
\begin{displaymath}
\Pi(x) = \sum_{m=1}^{M(x)} \frac{\pi(x^{1/m})}{m}.
\end{displaymath}
\end{defi}

In the following lemma, we note the following identity involving $\Pi(x)$, $R(x)$, $R_1(x)$, and $\text{li}(x)$.

\begin{lem} \label{lem202}
Let $y$ be a real number with $y \geq 2$. Then, one has
\begin{displaymath}
\Pi(x) - \emph{li}(x) = d_0(y) + \frac{R(x)}{\log x} + \frac{R_1(x)}{x\log^2 x} + \int_y^x \left( \frac{R_1(t)}{t^2\log^2 t} + \frac{R_1(t)}{t^2\log^3 t}\right) \, \emph{d}t
\end{displaymath}
for every $x \geq y$, where
\begin{displaymath}
d_0(y) = \Pi(y) - \emph{li}(y) - \frac{R(y)}{\log y} - \frac{R_1(y)}{y \log^2y}. \tag{2.5} \label{2.5}
\end{displaymath}
\end{lem}

\begin{proof}
See Ingham \cite[p.\:64]{ingham}.
\end{proof}

In the proof of Theorem \ref{thm102}, we will use the following upper and lower bound for the integral
\begin{equation}
\int_y^x (\Pi(t) - \text{li}(t)) \, \text{d}t, \tag{2.6} \label{2.6}
\end{equation}
where $x \geq y$.

\begin{prop} \label{prop203}
Let $c_0 = \log 2\pi$ and $c_1 = 2 + (\zeta'/\zeta)(-1)$. Let $\omega = \sum_{\rho} |\rho(\rho+1))|^{-1}$ where $\rho$ runs over all all nontrivial zeros of the Riemann zeta function. Further, let $\delta$ be a real positive number with $\delta \geq \omega$ and let $d_0(y)$ be defined as in \eqref{2.5}. If the Riemann hypothesis is true, then
\begin{displaymath}
j_1(y,\delta) + j_0(y,\delta)x - h(x, \delta) < \int_y^x (\Pi(t) - \emph{li}(t)) \, \emph{d}t < j_3(y, \delta) + j_2(y,\delta) x + h(x, \delta) 
\end{displaymath}
for every $x \geq y \geq 2$, where
\begin{align*}
h(x, \delta) & = \frac{3\delta}{4} (x \emph{li}(\sqrt{x}) - \emph{li}(x^{3/2})), \\
j_0(y, \delta) & = d_0(y) + \delta \left( \frac{3}{4} \, \emph{li}(\sqrt{y}) - \frac{\sqrt{y}(3\log y+2)}{2 \log^2 y} \right) - c_0 \left( \frac{1}{\log y} + \frac{1}{\log^2 y} \right), \\
j_1(y, \delta) & = - j_0(y, \delta)y - \frac{R_1(y)}{\log y} - \frac{c_0y + \delta y^{3/2}}{\log y} + \frac{3\delta}{4} (y \emph{li}(\sqrt{y}) - \emph{li}(y^{3/2})), \\
j_2(y,\delta) & = d_0(y) - \delta \left( \frac{3}{4} \, \emph{li}(\sqrt{y}) - \frac{\sqrt{y}(3\log y+2)}{2 \log^2 y} \right) - c_0 \left( \frac{1}{\log y} + \frac{1}{\log^2 y} \right) + \frac{c_1}{y\log^2y}, \\
j_3(y, \delta) & = -j_2(y, \delta) y - \frac{3\delta}{4} (y \emph{li}(\sqrt{y}) - \emph{li}(y^{3/2})) - \frac{R_1(y)}{\log y} + \frac{c_1 - c_0y + \delta y^{3/2}}{\log y}.
\end{align*}
\end{prop}

\begin{proof}
For a better readability, we set $c_0 = \log 2\pi$ and $c_1 = 2 + (\zeta'/\zeta)(-1)$. Using Lemma \ref{lem202}, we get
\begin{equation}
\int_y^x (\Pi(t) - \text{li}(t)) \, \text{d}t = \int_y^x \left( d_0(y) + \frac{R(t)}{\log t} + \frac{R_1(t)}{t\log^2 t} + \int_y^t \left( \frac{R_1(u)}{u^2 \log^2 u} + \frac{2R_1(u)}{u^2 \log^3 u} \right) \, \text{d}u \right) \, \text{d}t. \tag{2.7} \label{2.7}
\end{equation}
Let $\delta$ be a positive real number with $\delta \geq \omega$. Then, Lemma \ref{lem201} implies
\begin{equation}
R_1(x) \leq c_1 - c_0 x + \delta x^{3/2} \tag{2.8} \label{2.8}
\end{equation}
for every $x \geq 2$. If we apply the inequality \eqref{2.8} to \eqref{2.7}, we can see that
\begin{equation}
\int_y^t \left( \frac{R_1(u)}{u^2 \log^2 u} + \frac{2R_1(u)}{u^2 \log^3 u} \right) \, \text{d}u \leq d_2(t, \delta) - d_2(y, \delta), \tag{2.9} \label{2.9}
\end{equation}
where the function $d_2(u,\delta)$ is defined by
\begin{displaymath}
d_2(u, \delta) = \delta \left( \frac{3}{4} \, \text{li}(\sqrt{u}) - \frac{\sqrt{u}(3\log u+2)}{2 \log^2 u} \right) + c_0 \left( \frac{1}{\log u} + \frac{1}{\log^2 u} \right) - \frac{c_1}{u\log^2u}.
\end{displaymath}
Integration by parts and the inequality \eqref{2.8} together give
\begin{equation}
\int_y^x \left( \frac{R(t)}{\log t} + \frac{R_1(t)}{t\log^2 t} \right) \, \text{d}t \leq \delta \left( 3 \, \text{li}(x^{3/2}) - \frac{x^{3/2}}{\log x} \right) + c_0 \left( \frac{x}{\log x} - 2 \, \text{li}(x) \right) - \frac{c_1}{\log x} + d_3(y,\delta), \tag{2.10} \label{2.10}
\end{equation}
where
\begin{displaymath}
d_3(y,\delta) = - \frac{R_1(y)}{\log y} - \delta \left( 3 \, \text{li}(y^{3/2}) - \frac{2y^{3/2}}{\log s} \right) + 2c_0 \left( \text{li}(y) - \frac{y}{\log y} \right) + \frac{2c_1}{\log y}.
\end{displaymath}
If we substitute \eqref{2.9} and \eqref{2.10} into \eqref{2.7}, we obtain the required right-hand side inequality.
On the other hand, Lemma \eqref{lem201} implies that $R_1(x) \geq c_1 - c_0 x - x/(2(x^2-1)) - \delta x^{3/2}$ for every $x \geq y$. A simple calculation shows that $R_1(x) > - c_0x - \delta x^{3/2}$ for every $x \geq y$. Now we can proceed as in the first part of this proof to get the required left-hand side inequality.
\end{proof}

Proposition \ref{prop203} has the following direct consequence.

\begin{kor} \label{kor204}
We have
\begin{displaymath}
\limsup_{x \to \infty} \frac{\log x}{x^{3/2}} \int_2^x (\Pi(t) - \emph{li}(t)) \, \emph{d}t \leq \omega \q\emph{and} \q \liminf_{x \to \infty} \frac{\log x}{x^{3/2}} \int_2^x (\Pi(t) - \emph{li}(t)) \, \emph{d}t \geq - \omega.
\end{displaymath}
\end{kor}

\begin{proof}
First, we substitute $y=2$ and $\delta = \omega$ in Proposition \eqref{prop203}. Then, it suffices to note that $h(x, \omega) \sim \omega x^{3/2}/ \log x$
as $x \to \infty$.
\end{proof}

In order to give effective estimates for the integral
\begin{displaymath}
\int_2^x (\Pi(t) - \text{li}(t)) \, \text{d}t,
\end{displaymath}
we need do estimate the constant $\omega = \sum_{\rho} |\rho(\rho+1)|^{-1}$, where $\rho$ runs over all all nontrivial zeros of the Riemann zeta function. By \cite[Ch.\:12, Equations (10) and (11)]{davenport}, we have
\begin{equation}
\sum_{\rho} \frac{\text{Re}(\rho)}{|\rho|^2} = 1 + \frac{\gamma}{2} - \frac{\log 4\pi}{2}. \tag{2.11} \label{2.11}
\end{equation}
Under the assumption that the Riemann hypothesis is true, we have $\text{Re}(\rho) = 1/2$ and \eqref{2.11} implies
\begin{equation}
\omega \leq 2 + \gamma - \log 4\pi = 0.0461914179\ldots. \tag{2.12} \label{2.12}
\end{equation}
Now we can use \eqref{2.12} to find the following result concerning effective estimates for the integral given in \eqref{2.6} with $y = 2$ depending on the truth of the Riemann hypothesis.

\begin{kor} \label{kor205}
Under the assumption that the Riemann hypothesis is true, we have
\begin{displaymath}
-0.05 \times \frac{x^{3/2}}{\log x} < \int_2^x (\Pi(t) - \emph{li}(t)) \, \emph{d}t < 0.05 \times \frac{x^{3/2}}{\log x}
\end{displaymath}
for every $x \geq 1.15 \times 10^{16}$.
\end{kor}

\begin{proof}
Let $\lambda_0 = 2 + \gamma - \log 4\pi$. By \eqref{2.12}, we have $\lambda_0 \geq \omega$. So, if we substitute $\delta = \lambda_0$ and $y = 2$ in Proposition \ref{prop203}, we get
\begin{displaymath}
j_1 + j_0x - \frac{3\lambda_0}{4} (x \text{li}(\sqrt{x}) - \text{li}(x^{3/2})) < \int_2^x (\Pi(t) - \text{li}(t)) \, \text{d}t < j_3 + j_2 x + \frac{3\lambda_0}{4} (x \text{li}(\sqrt{x}) - \text{li}(x^{3/2}))
\end{displaymath}
for every $x \geq 2$, where $j_k = j_k(2,\lambda_0)$ for every integer $k$ with $0 \leq k \leq 3$. This gives the required upper and lower bound for every $x \geq 1.15 \times 10^{16}$.
\end{proof}

\begin{rema}
Note that the positive integer $N_0 = 1.15 \times 10^{16}$ might not be the smallest positive integer so that the left-hand side inequality given Corollary \ref{kor205} holds for every $x \geq N_0$.
\end{rema}

\begin{rema}
Since the function
\begin{displaymath}
\int_2^x (\Pi(t) - \text{li}(t)) \, \text{d}t
\end{displaymath}
changes its sign infinitely often as $x$ increases to infinity (see Johnston \cite[Lemma 2.8]{ingham}), a result similar to Proposition \ref{prop101} for $\Pi(t)$ instead of $\pi(t)$ does not hold.
\end{rema}

Now give a proof of Theorem \ref{thm102}.


\begin{proof}[Proof of Theorem \ref{thm102}]
For a better readability, we set $c_0 = \log 2\pi$ and $c_1 = 2 + (\zeta'/\zeta)(-1)$. First, we prove the required upper bound. Let $s = 17^2$ and let $x \geq s$. We have
\begin{equation}
\int_2^x (\pi(t) - \text{li}(t)) \, \text{d}t = \int_2^s (\pi(t) - \text{li}(t)) \, \text{d}t + \int_s^x (\pi(t) - \Pi(t)) \, \text{d}t + \int_s^x (\Pi(t) - \text{li}(t)) \, \text{d}t. \tag{2.13} \label{2.13}
\end{equation}
It is easy to see that $\Pi(t) \geq \pi(t) + \pi(\sqrt{t})/2$ for every $t \geq 2$. Rosser and Schoenfeld \cite[Corollary 1]{rosserschoenfeld1962} showed that $\pi(y) \geq y/\log y$ for every $y \geq 17$. Hence
\begin{equation}
\Pi(t) \geq \pi(t) + \frac{\sqrt{t}}{\log t} \tag{2.14} \label{2.14}
\end{equation}
for every $t \geq s$. Applying \eqref{2.14} to \eqref{2.13}, we can see that
\begin{equation}
\int_2^x (\pi(t) - \text{li}(t)) \, \text{d}t \leq \text{li}(s^{3/2}) + \int_2^s (\pi(t) - \text{li}(t)) \, \text{d}t - \text{li}(x^{3/2}) + \int_s^x (\Pi(t) - \text{li}(t)) \, \text{d}t. \tag{2.15} \label{2.15}
\end{equation}
Now we can apply the right-hand side inequality in Proposition \ref{prop203} with $y = s$ to get the required upper bound. Next, we verify the required left-hand side inequality. Next, we verify the required left-hand side inequality. In order to do this, let $x \geq 2$. Again, it is easy to see that $\Pi(t) \leq \pi(t) + \pi(\sqrt{t})/2 + \pi(t^{1/3})/3 + M(t) \pi(t^{1/4})/4$ for every $t \geq 2$. By Dusart \cite{dusart1999}, we have $\pi(t) < t/\log t + 1.2762t/\log^2 t$ for every $t > 1$. Further we can utilize \cite[Theorem 1.3]{axlernew} to see that $\pi(t) < 1.26t/\log t$ for every $t > 1$. Therefore,
\begin{displaymath}
\Pi(t) \leq \pi(t) + \frac{\sqrt{t}}{\log t} + \frac{2.5524 \sqrt{t}}{\log^2 t} + \frac{t^{1/3}}{\log t} + \frac{3.8286 t^{1/3}}{\log^2 t} + \frac{1.26t^{1/4}}{\log 2}
\end{displaymath}
for every $t \geq 2$. Similar to \eqref{2.15}, we get that
\begin{equation}
\int_2^x (\pi(t) - \text{li}(t)) \, \text{d}t > \text{li}(2^{3/2}) + g(2) - \text{li}(x^{3/2}) + g(x) + \int_2^x (\Pi(t) - \text{li}(t)) \, \text{d}t \tag{2.16} \label{2.16}
\end{equation}
where
\begin{displaymath}
g(x) = - 1.2762 \left( 3 \, \text{li}(x^{3/2}) - \frac{2x^{3/2}}{\log x} + 4 \, \text{li}(x^{4/3}) - \frac{3x^{4/3}}{\log x} \right) - \text{li}(x^{4/3}) - \frac{5.1048x^{5/4}}{5\log 2}.
\end{displaymath}
Now we can apply the left-hand side inequality given in Proposition \ref{prop203} with $y = 2$ to \eqref{2.16} and we see that the required lower bound is fulfilled for every $x \geq 2$.
\end{proof}

\begin{rema}
If we set $\delta = \omega$ and $P(t) = \pi(t) - \text{li}(t)$ in Theorem \ref{thm102}, we obtain under the assumption that the Riemann hypothesis is true that
\begin{equation}
\limsup_{x \to \infty} \left( \frac{\log x}{x^{3/2}} \int_2^x P(t) \, \text{d}t \right) \leq \omega-\frac{2}{3} \q \text{and} \q \liminf_{x \to \infty} \left( \frac{\log x}{x^{3/2}} \int_2^x P(t) \, \text{d}t \right) \geq -\omega-\frac{2}{3}. \tag{2.17} \label{2.17}
\end{equation}
This was already proven by Stechkin and Popov \cite[Theorem 13]{stechkin}.
\end{rema}

Now we can use \eqref{2.12} to find the following result concerning effective estimates for the integral given in Theorem \ref{thm102} depending on the truth of the Riemann hypothesis.

\begin{proof}[Proof of Corollary \ref{kor103}]
First, we assume that the Riemann hypothesis is true. If we set $\lambda_0 = 2 + \gamma - \log 4\pi$, we can use \eqref{2.12} to see that $\lambda_0 \geq \omega$. So, if we substitute $\delta = \lambda_0$ in Theorem \ref{thm102}, we get
\begin{displaymath}
J_1(\lambda_0) + J_0(\lambda_0)x + f_1(x,\lambda_0) + g(x) < \int_2^x (\pi(t) - \text{li}(t)) \, \text{d}t < J_3(\lambda_0) + J_2(\lambda_0) x + f_2(x, \lambda_0),
\end{displaymath}
where the left-hand side inequality holds for every $x \geq 2$ and the right-hand side inequality holds for every $x \geq s$. A simple computation shows that the required upper bound holds for every $x \geq 139,537,375$ and that the required lower bound is valid for every $x \geq 4.82 \times 10^{863}$.
To complete the proof it suffices to observe that the reverse implication follows directly from Proposition \ref{prop101}.
\end{proof}

\section{Proof of Theorem \ref{thm104}}

Euler \cite{euler1737} proved that the sum of the reciprocals of all prime numbers
diverges.
Mertens \cite[p.\:52]{mertens1874} found that $\log \log x$ is the right order of magnitude for this sum
by showing 
\begin{displaymath}
\sum_{p \leq x} \frac{1}{p} = \log \log x + B + O \left( \frac{1}{\log x} \right).
\end{displaymath}
Here $B$ denotes the Mertens' constant and is defined as in \eqref{1.9}. Setting
\begin{equation}
A_1(x) = \sum_{p \leq x} \frac{1}{p} - \log \log x - B, \tag{3.1} \label{3.1}
\end{equation}
the present author found the following effective estimates for $A_1(x)$.

\begin{lem} \label{lem301}
For every $x \geq 1,757,126,630,797$, we have
\begin{displaymath}
|A_1(x)| \leq \frac{0.024334}{3\log^3 x} \left( 1 + \frac{15}{4\log x} \right).
\end{displaymath}
\end{lem}

\begin{proof}
See \cite[Theorem 1.5]{axlernew}.
\end{proof}

In order to prove Theorem \ref{thm104}, we first note two more lemmata.

\begin{lem} \label{lem302}
Let the constant $C$ be defined as in \eqref{1.10}. Then, one has
\begin{equation}
\int_2^x \frac{\pi(t) - \emph{li}(t)}{t^2} \, \emph{d}t = \frac{\emph{li}(x) - \pi(x)}{x} + C + A_1(x) \tag{3.2} \label{3.2}
\end{equation}
for every $x \geq 2$.
\end{lem}

\begin{proof}
We can apply Abel's identity (see, for instance, \cite[Theorem 4.2]{ap}) with $f(x) = 1/x$ to get that
\begin{equation}
\sum_{p \leq x} \frac{1}{p} = \frac{\pi(x)}{x} + \int_2^x \frac{\pi(t)}{t^2} \, \text{d}t \tag{3.3} \label{3.3}
\end{equation}
for every $x \geq 2$. Using integration by parts, we can see that
\begin{equation}
\log \log x = \frac{\text{li}(x)}{x} + \int_2^x \frac{\text{li}(t)}{t^2} \, \text{d}t - \int_2^e \frac{\text{li}(t)}{t^2} \, \text{d}t - \frac{\text{li}(e)}{e} \tag{3.4} \label{3.4}
\end{equation}
for every $x \geq 2$. To complete the proof, it suffices to combine \eqref{3.3} with \eqref{3.4}.
\end{proof}

%
%

Next, we give the following explicit result conerning the distance between $\pi(x)$ and $\text{li}(x)$.

\begin{lem} \label{lem303}
We have
\begin{displaymath}
-\frac{0.024965x}{\log^4x} < \emph{li}(x) - \pi(x) < \frac{0.02711x}{\log^4x},
\end{displaymath}
where the left-hand side inequality holds for every $x \geq 2$ and the right-hand side inequality holds for every $x \geq 1,757,126,630,797$.
\end{lem}

\begin{proof}
Let $x_0 = 10^{19}$ and $y_0 = 1,423$. First, we consider the case where $x \geq x_0$. If we combine \eqref{1.15} and \cite[Corollary 11.1]{kadiri}, we get
\begin{displaymath}
\pi(x) < \text{li}(x) + \frac{0.024334x}{\log^4 x} - \text{li}(2) + \frac{2}{\log 2} + 0.024334(G(x) - G(2)),
\end{displaymath}
where
\begin{equation}
G(t) = \frac{1}{24} \left( \text{li}(t) - \frac{t(\log^3t + \log^2t + 2\log t + 6)}{\log^4t} \right). \tag{3.5} \label{3.5}
\end{equation}
Since $G(x) \leq 1.1331x/\log^5x$, we see that
\begin{displaymath}
\pi(x) < \text{li}(x) + \frac{0.024334x}{\log^4 x} \left( 1 + \frac{1.1331}{\log x} + \frac{1.9112 \log^4x}{0.024334x} \right)
\end{displaymath}
which gives the required left-hand side inequality for every $x \geq x_0$. For every $x$ with $2 \leq x \leq x_0$, it suffices to utilize \cite[Theorem 2]{buthe2018}. Now, we prove the required right-hand side inequality. Again, we first consider the case where $x \geq x_0$. By Büthe \cite[Theorem 2]{buthe2018}, we have
\begin{equation}
\vartheta(t) > t - 1.95\sqrt{t} \tag{3.6} \label{3.6}
\end{equation}
for every $t$ with $y_0 \leq t \leq x_0$. Combining \eqref{1.15}, \eqref{3.6}, and \cite[Proposition 1.1]{axlernew}, it turns out that
\begin{displaymath}
\pi(x) > q_0 + \text{li}(x) - \frac{0.024334x}{\log^4 x} - 0.024334 G(x),
\end{displaymath}
where $G(x)$ is defined as in \eqref{3.5} and the constant $q_0$ is given by
\begin{displaymath}
q_0 = \int_2^{y_0} \frac{\vartheta(t)}{t \log^2t} \, \text{d}t - \text{li}(y_0) + \frac{y_0}{\log y_0} - 1.95 \left( \frac{\text{li}(\sqrt{x_0})}{2} - \frac{\sqrt{x_0}}{\log x_0} - \frac{\text{li}(\sqrt{y_0})}{2} + \frac{\sqrt{y_0}}{\log y_0}\right) + 0.024334 G(x_0).
\end{displaymath}
Since $q_0 > 1.7 \times 10^9$, we can argue as in the first part of the proof to get that
\begin{displaymath}
\pi(x) > \text{li}(x) - \frac{0.024334x}{\log^4 x} \left( 1 + \frac{1.1331}{\log x} \right)
\end{displaymath}
which implies the required right-hand side inequality for every $x \geq x_0$. So it remains to consider the case where $1,757,126,630,797 \leq x \leq x_0$. We use again the inequality \eqref{3.6} to obtain that
\begin{equation}
\pi(x) > q_1 + \text{li}(x) - \frac{0.024334x}{\log^4 x} - 1.95 \left( \frac{\text{li}(\sqrt{x})}{2} - \frac{\sqrt{x}}{\log x} \right), \tag{3.8} \label{3.7}
\end{equation}
where $q_1$ is defined as
\begin{displaymath}
q_1 = \int_2^{y_0} \frac{\vartheta(t)}{t \log^2t} \, \text{d}t - \text{li}(y_0) + \frac{y_0}{\log y_0} + 1.95 \left( \frac{\text{li}(\sqrt{y_0})}{2} - \frac{\sqrt{y_0}}{\log y_0} \right).
\end{displaymath}
Note that $q_1 > 1.1254$. Hence, the inequality \eqref{3.7} implies the required right-hand side inequality for every $x$ with $1,757,126,630,797 \leq x \leq x_0$.
%
%
\end{proof}

Now the proof of Theorem \ref{thm104} is quite simple.

\begin{proof}[Proof of Theorem \ref{thm104}]
It suffices to combine \eqref{3.2} with Lemmata \ref{lem301} and \ref{lem303}.
\end{proof}

\section{Proof of Theorem \ref{thm105}}

In 1874, Mertens \cite{mertens1874} showed that
\begin{equation}
\sum_{p \leq x} \frac{\log p}{p} = \log x + O(1). \tag{4.1} \label{4.1}
\end{equation}
Landau \cite[§55]{landau} improved \eqref{4.1} by finding
\begin{displaymath}
\sum_{p \leq x} \frac{\log p}{p} = \log x + E + O (\exp(-\sqrt[14]{\log x})),
\end{displaymath}
where $E$ is the constant defined as in \eqref{1.14}. Similar to \eqref{3.1}, we set
\begin{equation}
A_2(x) = \sum_{p \leq x} \frac{\log p}{p} - \log x - E \tag{4.2} \label{4.2}
\end{equation}
and note the following effective estimates for $A_2(x)$.

\begin{lem} \label{lem401}
For every $x \geq 1,757,126,630,797$, we have
\begin{displaymath}
|A_2(x)| \leq \frac{0.024334}{2\log^2x} \left( 1 + \frac{2}{\log x} \right).
\end{displaymath}
\end{lem}

\begin{proof}
See \cite[Theorem 1.6]{axlernew}.
\end{proof}

In order to prove Theorem \ref{thm105}, we proceed as in the proof of Theorem \ref{thm104} and first show the following

\begin{prop} \label{prop402}
Let $A_1(x)$ and $A_2(x)$ be defined as in \eqref{3.1} and \eqref{4.2}, respectively. Further, we set
\begin{displaymath}
K = B + E + \log 2 + \log \log 2 - \frac{(1+\log 2)\emph{li}(2)}{2} = -1.62925885667\ldots,
\end{displaymath}
where the constants $B$ and $E$ are defined as in \eqref{1.9} and \eqref{1.14}. Then one has
\begin{displaymath}
\int_2^x \frac{(\pi(t) - \emph{li}(t)) \log t}{t^2} \, \emph{d}t = K + A_1(x) + A_2(x) + \frac{(\emph{li}(x) - \pi(x))(\log x +1)}{x}
\end{displaymath}
for every $x \geq 2$.
\end{prop}

\begin{proof}
By Rosser and Schoenfeld \cite[p.\:67]{rosserschoenfeld1962}, we get
\begin{displaymath}
\int_2^x \frac{(\pi(t) - \text{li}(t)) \log t}{t^2} \, \text{d}t = \sum_{p \leq x} \left( \frac{1}{p} + \frac{\log p}{p} \right) - \frac{\pi(x)(\log x +1)}{x} - \int_2^x \frac{\text{li}(t) \log t}{t^2} \, \text{d}t
\end{displaymath}
for every $x \geq 2$. Now we can proceed as in the proof of Lemma \ref{lem302} to complete the proof.
\end{proof}

Now it is rather simple to give a proof of Theorem \ref{thm105}.

\begin{proof}[Proof of Theorem \ref{thm105}]
We combine Proposition \ref{prop402}, Lemmata \ref{lem301} and \ref{lem401}, and Lemma \ref{lem303}.
\end{proof}

Finally, we can utilize Theorem \ref{thm105} to prove Corollary \ref{kor106} where we state that that the integral Proposition \ref{prop402} is indeed negative for every $x > 2$.

\begin{proof}[Proof of Corollary \ref{kor106}]
If $x \geq 1,757,126,630,797$, the claim follows immediately from Theorem \ref{thm105}. If $x$ satisfies $2 < x < 1,757,126,630,797$, we can use \cite[Theorem 2]{buthe2018} to complete the proof. 
\end{proof}

\section{Proof of Theorem \ref{thm107}}

Remember that Johnston \cite[Theorem 1.2]{johnston2} also found the following further equivalent criterion for the truth of the Riemann hypothesis.

\begin{lem}[Johnston] \label{lem501}
The Riemann hypothesis is true if and only if
\begin{displaymath}
\int_2^x (\vartheta(t) - t) \, \emph{d}t < 0
\end{displaymath}
for every $x > 2$.
\end{lem}

Now, we can use a result of Nicolas \cite[Lemma 2.4]{nicolas} to give the following proof of Theorem \ref{thm107} where we give a slight refinement of Johnston's criterion in Lemma \ref{lem501}.

\begin{proof}[Proof of Theorem \ref{thm107}]
We assume that the Riemann hypothesis is true and let $s = 121$. First, we note that
\begin{equation}
\int_2^x (\vartheta(t) - t) \, \text{d}t = \int_2^x (\vartheta(t) - \psi(t)) \, \text{d}t + R_1(x) \tag{5.1} \label{5.1}
\end{equation}
for every $x \geq 2$, where $R_1(x)$ is defined as in \eqref{2.1}. As in the proof of Theorem \ref{thm102}, let $c_0 = \log 2\pi$ and $c_1 = 2 + (\zeta'/\zeta)(-1)$. By Nicolas \cite[Lemma 2.4]{nicolas}, we have $\psi(t) - \vartheta(t) > \sqrt{t}$ for every $t \geq s$. Substituting this inequality and the right-hand side inequality of Lemma \ref{lem201} into \eqref{5.1}, we can see that
\begin{displaymath}
\int_2^x (\vartheta(t) - t) \, \text{d}t < \mu_0 + \left( \omega - \frac{2}{3} \right) x^{3/2} - c_0 x
\end{displaymath}
for every $x \geq s$, where the constant $\mu_0$ is defined by
\begin{displaymath}
\mu_0 = \int_2^s (\vartheta(t) - t) \, \text{d}t + c_1 + \frac{2s^{3/2}}{3}.
\end{displaymath}
Since $\mu_0 = -210.2527\ldots$, we get the required right-hand side inequality for every $x \geq s$. On the other hand, we have $\psi(t) - \vartheta(t) < \sqrt{t} + 4x^{1/3}/3$ for every $t \geq 1$ (see Nicolas \cite[Lemma 2.4]{nicolas}). Applying this inequality toghether with the left-hand side inequality of Lemma \ref{lem201} to \eqref{5.1}, it turns out that
\begin{displaymath}
\int_2^x (\vartheta(t) - t) \, \text{d}t > \mu_1 - \left( \omega + \frac{2}{3} \right) \, x^{3/2} - x^{4/3} - c_0x - \frac{x}{2(x^2-1)}
\end{displaymath}
for every $x \geq 2$, where $\mu_1$ is a constant defined by $\mu_1 = c_1 + 2^{4/3} + 2^{5/2}/3$. A simple computation shows that $\mu_1 > t/(2(t^2-1))$ for every $t \geq 2$. Hence, the required left-hand side inequality is fulfilled for every $x \geq 2$ and we arrive at the end of the proof.
\end{proof}

As a direct consequence of Theorem \ref{thm107} we get the following analogy of \eqref{2.17}.

\begin{kor} \label{kor1102}
Let $\omega = \sum_{\rho} |\rho(\rho+1))|^{-1}$ where $\rho$ runs over all all nontrivial zeros of the Riemann zeta function. Under the assumption that the Riemann hypothesis is true, one has
\begin{displaymath}
\limsup_{x \to \infty} \left( \frac{1}{x^{3/2}} \int_2^x (\vartheta(t) - t) \, \emph{d}t \right) \leq \omega -\frac{2}{3}
\end{displaymath}
and
\begin{displaymath}
\liminf_{x \to \infty} \left( \frac{1}{x^{3/2}} \int_2^x (\vartheta(t) - t) \, \emph{d}t \right) \geq -\omega - \frac{2}{3}.
\end{displaymath}
\end{kor}

Fianlly, we give a proof of Corollary \ref{kor108}.

\begin{proof}[Proof of Corollary \ref{kor108}]
If we assume that the Riemann hypothesis it true, we can combine the inequalities obtained in Theorem \ref{thm107} with the upper bound \eqref{2.12} for $\omega$ to get the required upper and lower bound. Finally, it suffices to note that Lemma \ref{lem501} implies the reverse implication.
\end{proof}

\section{Proof of Theorem \ref{thm109}}

In order to prove Theorem \ref{thm109} where we give an improvement of \eqref{1.19}, we first note the following lemma.

\begin{lem} \label{lem601}
Let $A_2(x)$ and the constant $D$ be given as in \eqref{4.2} and \eqref{1.20}, respectively. Then, we have
\begin{displaymath}
\int_2^x \frac{\vartheta(t) - t}{t^2} \, \emph{d}t = D - \frac{\vartheta(x)-x}{x} + A_2(x)
\end{displaymath}
for every $x \geq 2$.
\end{lem}

\begin{proof}
Similar to the proof of Lemma \ref{lem302}.
\end{proof}

Because of Lemma \ref{lem601}, the proof of Theorem \eqref{thm109} is now rather simple.

%

\begin{proof}[Proof of Theorem \ref{thm109}]
The claim follows from Lemmata \ref{lem601} and \ref{lem401}, and \cite[Proposition 1.1]{axlernew}.
\end{proof}

\begin{kor} \label{kor602}
One has
\begin{displaymath}
\int_2^{\infty} \frac{\psi(t) - \vartheta(t)}{t^2} \, \emph{d}t = \sum_{p} \frac{\log p}{p(p-1)} = 0.755366 \ldots.
\end{displaymath}
\end{kor}

\begin{proof}
Let $\gamma$ be the Euler-Mascheroni constant. By Jameson \cite[Proposition 3.4.4]{jameson}, we have
\begin{equation}
\int_2^{\infty} \frac{\psi(t) - t}{t^2} \, \text{d}t = - \gamma - 1 + \log 2 = -0.884068\ldots. \tag{6.1} \label{6.1}
\end{equation}
If we combine this equality with \eqref{1.21}, \eqref{1.20}, and \eqref{1.14}, we get the desired identity.
\end{proof}

By \eqref{6.2}, there is a smallest positive integer $x_0$ so that
\begin{displaymath}
 \int_2^x \frac{\psi(t) - t}{t^2} \, \text{d}t < 0
\end{displaymath}
for every $x > x_0$. Johnston \cite[Theorem 1.3]{johnston2} showed that $x_0 = 2$. Now, Theorem \ref{thm109} and two results of Dusart \cite{dusart20181} imply the following better result.

\begin{kor} \label{kor603}
For every $x \geq 1,757,126,630,797$, one has
\begin{displaymath}
 - 0.8894 - \frac{0.024334}{2\log^2 x} \left( 1 + \frac{4}{\log x} \right) < \int_2^x \frac{\psi(t) - t}{t^2} \, \emph{d}t < - 0.8802 + \frac{0.024334}{2\log^2 x} \left( 1 + \frac{4}{\log x} \right).
\end{displaymath}
.
\end{kor}

\begin{proof}
For a better readability, we set $s = 5,000$, $a_0 = 1+1.47 \times 10^{-7}$, and $a_1 = 0.9999$. One has
\begin{equation}
\int_2^x \frac{\psi(t) - t}{t^2} \, \text{d}t = \int_2^s \frac{\psi(t) - \vartheta(t)}{t^2} \, \text{d}t + \int_s^x \frac{\psi(t) - \vartheta(t)}{t^2} \, \text{d}t + \int_2^x \frac{\vartheta(t) - t}{t^2} \, \text{d}t. \tag{6.2} \label{6.2}
\end{equation}
By Dusart \cite[Corollary 4.5]{dusart20181}, we have $\psi(t) - \vartheta(t) < a_0\sqrt{t} + 1.78t^{1/3}$ for every $t > 0$. If we substitute the last inequality and the right-hand side inequality given in Theorem \ref{thm109} into \eqref{6.2}, we obtain the required upper bound
for every $x \geq 1,757,126,630,797$. On the other hand, Dusart \cite[Proposition 4.3]{dusart20181} showed that $\psi(t) - \vartheta(t) > a_1\sqrt{t}$ for every $t \geq 121$. Together with the left-hand side inequality given in Theorem \ref{thm109} and \eqref{6.2}, we obtain the required lower bound for every $x \geq 1,757,126,630,797$.
\end{proof}

\begin{rema}
Note that the positive integer $N_0 = 1,757,126,630,797$ might not be the smallest positive integer so that the inequalities in Corollary \ref{kor603} holds for every $x \geq N_0$.
\end{rema}

\section*{Acknowledgement}
I would like to thank Daniel R. Johnston, whose paper have inspired me to deal with the present subject. Moreover, I would also like to thank the two beautiful souls R. and O. for the never ending inspiration. Finally, I thank the anonymous reviewer for the useful comments and suggestions to improve the quality of this paper.

\end{document}